\documentclass[pdflatex,sn-mathphys-num]{sn-jnl}

\usepackage{amsmath,amssymb,amsfonts}
\usepackage{amsthm}
\usepackage{graphicx}
\usepackage{silence}
\PassOptionsToPackage{compatibility=false}{caption}
\usepackage{subcaption}
\usepackage{tikz}

\theoremstyle{thmstyleone}

\newtheorem{prop}{Proposition}
\newtheorem{lem}{Lemma}

\theoremstyle{thmstylethree}

\newtheorem*{theoremA}{Theorem A}

\begin{document}

\title % [Short title for header]
{
Deterministic Structure of Vertical
Configurations in Minimal Picker Tours
for Rectangular Warehouses
}

\author*[1]{\fnm{George} \sur{Dunn}}\email{george.dunn@uon.edu.au}
\author[1]{\fnm{Elizabeth} \sur{Stojanovski}}
\author[1]{\fnm{Bishnu} \sur{Lamichhane}}
\author[2]{\fnm{Hadi} \sur{Charkhgard}}
\author[3,4]{\fnm{Ali} \sur{Eshragh}}

\affil*[1]{\orgdiv{School of Information and Physical Sciences}, \orgname{University of Newcastle}, \orgaddress{\state{NSW}, \country{Australia}}}
\affil[2]{\orgdiv{Department of Industrial and Management Systems Engineering}, \orgname{University of South Florida}, \orgaddress{\state{FL}, \country{USA}}}
\affil[3]{\orgdiv{Carey Business School}, \orgname{Johns Hopkins University}, \orgaddress{\state{MD}, \country{USA}}}
\affil[4]{\orgdiv{International Computer Science Institute}, \orgname{University of California at Berkeley}, \orgaddress{\state{CA}, \country{USA}}}

\abstract{
The picker routing problem seeks the shortest tour
through a warehouse that visits every item
in a given pick-list and returns to the depot.
For rectangular warehouses,
dynamic programming algorithms solve this problem
by sequentially evaluating combinations of
vertical edge configurations within subaisles
and horizontal edge configurations between aisles.
These methods proceed through stages one after another,
but how those stages
relate to each other has received limited
structural analysis.
Building on our recent structural result for rectangular
warehouses, which shows that connecting
double traversals are not required to maintain tour
connectivity, we prove that for rectangular warehouses of any size, the horizontal edge
structure of a minimal tour subgraph uniquely
determines the required vertical edge configurations.
The proof uses a case analysis on horizontal degree along each aisle and at
merged-segment endpoints, showing that the admissible vertical pattern in each
regime is uniquely determined by Eulerian parity and by minimizing traversal
length.
This deterministic relationship implies that
vertical configuration stages in existing
dynamic programming algorithms can be replaced
by a direct inference step,
reducing the combinatorial complexity
of the problem and providing a structural
foundation for developing more efficient
exact methods for warehouse layouts of any size.
}



\maketitle

% ~~~~~~~~~~~~~~~~~~~~~~~~~~~~~~~~~~~~~~~~~~~~~~~~~~~~~~~~~~~~~~~~~

\section{Introduction}
\label{sec:introduction}

Order picking is the process of collecting goods
in a warehouse to fulfill customer orders.
The picker routing problem seeks the shortest tour that
visits all required item locations and returns to a depot.
For single-block parallel-aisle warehouses,
\citet{ratliff1983order} proposed a dynamic programming algorithm,
later extended to two-block layouts by
\citet{roodbergen2001routing} and to general multi-block
warehouses by \citet{pansart2018exact}.
These algorithms construct optimal tours by sequentially
evaluating combinations of vertical edge configurations within
subaisles and horizontal connections between aisles.

We consider a rectangular warehouse with
a single depot,
$m \geq 1$ vertical aisles, and
$n \geq 2$ horizontal cross-aisles.
As the warehouse is rectangular,
the distances between adjacent aisles
and between adjacent cross-aisles are uniform.
The cross-aisles divide each aisle into subaisles,
which contain the stored items and
are assumed to be sufficiently
narrow such that the horizontal distance
to traverse them is negligible.
The warehouse can be represented as a graph
$G = (V \cup P, E)$
as illustrated in Figure~\ref{fig:figure_warehouse_graph},
with vertices $v_{i,j} \in V$
at the intersection of aisle $i \in \{0,\ldots,m-1\}$ and
cross-aisle $j \in \{0,\ldots,n-1\}$,
respectively.
The set of vertices,
$P = \{p_0, p_1, \ldots, p_q\}$,
represents the locations to be visited,
with $p_0$ as the depot and
$p_1, \ldots, p_q$ the products to be collected.
Figure~\ref{fig:warehouse} illustrates an instance with $q=9$.
Only $p_0$ can be located at a $v_{i,j}$ vertex,
while all other vertices of $P$
are located within the subaisles.

\begin{figure}[ht]
    \centering
        \begin{minipage}{0.48\textwidth}
            \centering
            \resizebox{\textwidth}{!}
            {
                \begin{tikzpicture}[shorten >=1pt,draw=black!50]

                    \pgfmathsetmacro{\x}{2.5}
                    \pgfmathsetmacro{\y}{1.2}
                    \pgfmathsetmacro{\nodesize}{26}
                
                    \draw[thick, double, double distance between line centers=5pt]
                    (0, 0) -- (3 * \x, 0);
                
                    \draw[thick, double, double distance between line centers=5pt]
                    (0, 4 * \y) -- (3 * \x, 4 * \y);
                
                    \draw[thick, double, double distance between line centers=5pt]
                    (0, 8 * \y) -- (3 * \x, 8 * \y);
                
                    \foreach \name / \a in {0,...,3}{
                        \node[shape=circle,draw=black, minimum size=\nodesize pt, fill=white]
                        (A-\name) at (\x * \a, 8 * \y) {};
                        \node[align=center] at (A-\name) {$v_{\name, 2}$};
                        
                        \node[shape=circle,draw=black, minimum size=\nodesize pt, fill=white]
                        (B-\name) at (\x * \a, 4 * \y) {};
                        \node[align=center] at (B-\name) {$v_{\name, 1}$};
                
                        \node[shape=circle,draw=black, minimum size=\nodesize pt, fill=white]
                        (C-\name) at (\x * \a, 0) {};
                        \node[align=center] at (C-\name) {$v_{\name, 0}$};
                
                        \draw[thick, double, double distance between line centers=5pt]
                        (A-\name) -- (B-\name) -- (C-\name);
                        }
                
                    \node[align=center] at (0.55, 0.55) {$p_0$};
                
                    \node[circle, draw=black, minimum size=\nodesize pt, fill=white]
                    (item1) at (0, 1 * \y) {};
                    \node[align=center] at (item1) {$p_1$};
                
                    \node[circle, draw=black, minimum size=\nodesize pt, fill=white]
                    (item2) at (0, 6 * \y) {};
                    \node[align=center] at (item2) {$p_2$};
                
                    \node[circle, draw=black, minimum size=\nodesize pt, fill=white]
                    (item3) at (1 * \x, 1 * \y) {};
                    \node[align=center] at (item3) {$p_3$};
                
                    \node[circle, draw=black, minimum size=\nodesize pt, fill=white]
                    (item4) at (1 * \x, 3 * \y) {};
                    \node[align=center] at (item4) {$p_4$};
                
                    \node[circle, draw=black, minimum size=\nodesize pt, fill=white]
                    (item5) at (1 * \x, 6 * \y) {};
                    \node[align=center] at (item5) {$p_5$};
                
                    \node[circle, draw=black, minimum size=\nodesize pt, fill=white]
                    (item6) at (2 * \x, 1 * \y) {};
                    \node[align=center] at (item6) {$p_6$};
                
                    \node[circle, draw=black, minimum size=\nodesize pt, fill=white]
                    (item7) at (2 * \x, 5 * \y) {};
                    \node[align=center] at (item7) {$p_7$};
                
                    \node[circle, draw=black, minimum size=\nodesize pt, fill=white]
                    (item8) at (2 * \x, 7 * \y) {};
                    \node[align=center] at (item8) {$p_8$};
                
                    \node[circle, draw=black, minimum size=\nodesize pt, fill=white]
                    (item9) at (3 * \x, 7 * \y) {};
                    \node[align=center] at (item9) {$p_9$};
                
                    \node (left) at (-0.5 * \x, 0) { };
                    \node (right) at (3.5 * \x, 9 * \y) { };
                
                    \node[align=center, rotate=90] (aisle) at (-0.4 * \x, 4 * \y) {\large Aisle};
                    \draw[->] (aisle) -- (-0.4 * \x, 0);
                    \draw[->] (aisle) -- (-0.4 * \x, 8 * \y);
                
                    \node[align=center, rotate=90] (subaisle) at (3.4 * \x, 2 * \y) {\large Subaisle};
                    \draw[->] (subaisle) -- (3.4 * \x, 0);
                    \draw[->] (subaisle) -- (3.4 * \x, 4 * \y);
                
                    \node[align=center] (cross) at (1.5 * \x, 8.75 * \y) {\large Cross-Aisle};
                    \draw[->] (cross) -- (0, 8.75 * \y);
                    \draw[->] (cross) -- (3 * \x, 8.75 * \y);
                
                \end{tikzpicture}
            }
            \subcaption{Warehouse graph $G$}   
            \label{fig:figure_warehouse_graph}
        \end{minipage}
        \hfill
        \begin{minipage}{0.48\textwidth}
            \centering
            \resizebox{\textwidth}{!}
            {
                \begin{tikzpicture}[shorten >=1pt,draw=black!50]

                    \pgfmathsetmacro{\x}{2.5}
                    \pgfmathsetmacro{\y}{1.2}
                    \pgfmathsetmacro{\nodesize}{26}
                
                    \foreach \name / \a in {0,...,3}{
                        \node[shape=circle,draw=black, minimum size=\nodesize pt, fill=white]
                        (A-\name) at (\x * \a, 8 * \y) {};
                        \node[align=center] at (A-\name) {$v_{\name, 2}$};
                        
                        \node[shape=circle,draw=black, minimum size=\nodesize pt, fill=white]
                        (B-\name) at (\x * \a, 4 * \y) {};
                        \node[align=center] at (B-\name) {$v_{\name, 1}$};
                
                        \node[shape=circle,draw=black, minimum size=\nodesize pt, fill=white]
                        (C-\name) at (\x * \a, 0) {};
                        \node[align=center] at (C-\name) {$v_{\name, 0}$};
                
                        }
                
                    \node[align=center] at (0.55, 0.55) {$p_0$};
                
                    \node[circle, draw=black, minimum size=\nodesize pt, fill=white]
                    (item1) at (0, 1 * \y) {};
                    \node[align=center] at (item1) {$p_1$};
                
                    \node[circle, draw=black, minimum size=\nodesize pt, fill=white]
                    (item2) at (0, 6 * \y) {};
                    \node[align=center] at (item2) {$p_2$};
                
                    \node[circle, draw=black, minimum size=\nodesize pt, fill=white]
                    (item3) at (1 * \x, 1 * \y) {};
                    \node[align=center] at (item3) {$p_3$};
                
                    \node[circle, draw=black, minimum size=\nodesize pt, fill=white]
                    (item4) at (1 * \x, 3 * \y) {};
                    \node[align=center] at (item4) {$p_4$};
                
                    \node[circle, draw=black, minimum size=\nodesize pt, fill=white]
                    (item5) at (1 * \x, 6 * \y) {};
                    \node[align=center] at (item5) {$p_5$};
                
                    \node[circle, draw=black, minimum size=\nodesize pt, fill=white]
                    (item6) at (2 * \x, 1 * \y) {};
                    \node[align=center] at (item6) {$p_6$};
                
                    \node[circle, draw=black, minimum size=\nodesize pt, fill=white]
                    (item7) at (2 * \x, 5 * \y) {};
                    \node[align=center] at (item7) {$p_7$};
                
                    \node[circle, draw=black, minimum size=\nodesize pt, fill=white]
                    (item8) at (2 * \x, 7 * \y) {};
                    \node[align=center] at (item8) {$p_8$};
                
                    \node[circle, draw=black, minimum size=\nodesize pt, fill=white]
                    (item9) at (3 * \x, 7 * \y) {};
                    \node[align=center] at (item9) {$p_9$};
                
                    \draw[-, draw=black, thick]
                    (C-0) -- (item1) -- (B-0) -- (B-1) -- (item5) --
                    (A-1) -- (A-2) -- (item8) -- (item7) --
                    (B-2) -- (item6) -- (C-2) -- (C-1) -- (C-0);
                
                    \draw[-, draw=black, thick, double, double distance between line centers=5pt]
                    (B-0) -- (item2);
                
                    \draw[-, draw=black, thick, double, double distance between line centers=5pt]
                    (B-1) -- (item4);
                
                    \draw[-, draw=black, thick, double, double distance between line centers=5pt]
                    (C-1) -- (item3);
                
                    \draw[-, draw=black, thick, double, double distance between line centers=5pt]
                    (A-2) -- (A-3) -- (item9);
                
                    \node (left) at (-0.5 * \x, 0) { };
                    \node (right) at (3.5 * \x, 9 * \y) { };
                
                \end{tikzpicture}
            }
            \subcaption{Tour subgraph $T$}   
            \label{fig:figure_warehouse_tour}
        \end{minipage}
    \caption{A rectangular warehouse with $m=4$ aisles, $n=3$ cross-aisles,
    and pick locations $p_0, \ldots, p_9$.
    (a)~The warehouse represented as a graph $G = (V \cup P, E)$ and
    (b)~an example tour subgraph $T$}
    \label{fig:warehouse}
    \end{figure}

A subgraph $T \subset G$ is a tour subgraph if it contains
all vertices $p_i \in P$ and there exists an order picking
tour that uses each edge in $T$ exactly once.
Figure~\ref{fig:figure_warehouse_tour} shows an example of
a tour subgraph for the graph
presented in Figure~\ref{fig:figure_warehouse_graph}.
The problem of finding an optimal order picking tour
can therefore be solved by finding a tour subgraph
with the minimum total edge length.
The following theorem defines the characteristics of a tour subgraph
\citep{ratliff1983order, christofides1975graph}.

\begin{theoremA}[Ratliff and Rosenthal, 1983]
\label{thm:subtour}
A subgraph $T \subseteq G$
is a tour subgraph if and only if:
\begin{enumerate}[(i)]
    \item Every vertex in $P$ is a vertex of $T$;
    \item $T$ is connected; and 
    \item Every vertex in $T$ has an even degree.
\end{enumerate}
\end{theoremA}

In this paper,
we show that in rectangular warehouses,
the choice of horizontal edges alone suffices
to determine the minimal set of vertical configurations
in a minimum-length tour.
This result relies on our recent finding that, in rectangular
warehouses, a double traversal
is never required to maintain connectivity of a minimal tour subgraph
\citep{dunn2025double}.
As a result, once the horizontal structure is fixed, the Eulerian and
minimality constraints uniquely pin down the vertical configuration.
This structural result eliminates the need to explore
vertical and horizontal configurations jointly,
reducing the combinatorial complexity of the problem.

Beyond computational speedups, deterministic structural descriptions are
useful for both algorithm design and rigorous modeling.
When the horizontal edge structure fixes the vertical structure in minimal
tours, the remaining decisions can be formulated with fewer degrees of
freedom; dynamic programs and exact enumeration can avoid jointly branching
over vertical patterns; and mathematical programming models can encode the
constraint structure directly \citep{goeke2021modeling}.
Such structural constraints can also support correctness proofs,
help derive lower bounds, and connect warehouse routing to broader
Eulerian and traveling-salesman-type graph phenomena.

Our technical contributions are Lemma~\ref{lem:merging}, which extends the six
vertical patterns to merged subaisle segments, and
Proposition~\ref{prop:determinacy}, which shows that horizontal edges incident to
an aisle uniquely determine its vertical configuration in a minimal tour
subgraph of a rectangular warehouse (building on the exclusion of connecting
double traversals from \citet{dunn2025double}).
Section~\ref{sec:preliminaries} introduces the
edge configurations, notation, and a subaisle merging lemma
needed for the analysis.
Section~\ref{sec:main_result} presents the main result
and its proof.
Section~\ref{sec:conclusion} discusses the implications
and concludes this paper.

% ~~~~~~~~~~~~~~~~~~~~~~~~~~~~~~~~~~~~~~~~~~~~~~~~~~~~~~~~~~~~~~~~~

\section{Preliminaries}
\label{sec:preliminaries}

For a minimal tour subgraph,
there are only six possible vertical edge configurations
within each subaisle and three horizontal edge configurations
between adjacent aisles at each cross-aisle
(no edge, one edge, or two parallel edges),
as shown in Figure~\ref{fig:vertical_action}
and Figure~\ref{fig:horizontal_action},
respectively
\citep{ratliff1983order, roodbergen2001routing, pansart2018exact}.
\citet{revenant2025note} showed that full double traversals
of a subaisle are
not required for single-block warehouses.
For rectangular warehouses of any size,
\citet{dunn2025double} demonstrated that double traversals
are not required to
maintain connectivity of a minimal tour subgraph.

\begin{figure}[ht]
    \centering
        \begin{minipage}{0.49\textwidth}
            \centering
            \resizebox{\textwidth}{!}
            {
                \begin{tikzpicture}[shorten >=1pt,->,draw=black!50, node distance=1.2cm]

                    \tikzset{minimum size=34pt}
                    
                        \node[shape=circle,draw=black] (a1) at (0, 5) { };
                        \node at (a1) {$v_{i,j+1}$};
                        \node[shape=circle,draw=black] (d1) at (0, 3.5) { };
                        \node[shape=circle,draw=black] (c1) at (0, 1.5) { };
                        \node[shape=circle,draw=black] (b1) at (0, 0) {$v_{i,j}$}; 
                        \draw[-, thick, draw=black] (a1) -- (d1) -- (c1) -- (b1);
                    
                        \node (i) at (0, -1) {(i)};
                    
                        \node[shape=circle,draw=black] (a2) at (1.5, 5) { };
                        \node at (a2) {$v_{i, j+1}$};
                        \node[shape=circle,draw=black] (d2) at (1.5, 3.5) { };
                        \node[shape=circle,draw=black] (c2) at (1.5, 1.5) { };
                        \node[shape=circle,draw=black] (b2) at (1.5, 0) {$v_{i, j}$}; 
                        \draw[-, thick, draw=black, double, double distance between line centers=8pt] (a2) -- (d2) -- (c2);
                    
                        \node (ii) at (1.5, -1) {(ii)};
                    
                        \node[shape=circle,draw=black] (a3) at (3, 5) { };
                        \node at (a3) {$v_{i, j+1}$};
                        \node[shape=circle,draw=black] (d3) at (3, 3.5) { };
                        \node[shape=circle,draw=black] (c3) at (3, 1.5) { };
                        \node[shape=circle,draw=black] (b3) at (3, 0) {$v_{i, j}$}; 
                        \draw[-, thick, draw=black, double, double distance between line centers=8pt] (b3) -- (c3) -- (d3);
                    
                        \node (iii) at (3, -1) {(iii)};
                    
                        \node[shape=circle,draw=black] (a4) at (4.5, 5) { };
                        \node at (a4) {$v_{i, j+1}$};
                        \node[shape=circle,draw=black] (d4) at (4.5, 3.5) { };
                        \node[shape=circle,draw=black] (c4) at (4.5, 1.5) { };
                        \node[shape=circle,draw=black] (b4) at (4.5, 0) {$v_{i, j}$}; 
                        \draw[-, thick, draw=black, double, double distance between line centers=8pt] (a4) -- (d4);
                        \draw[-, thick, draw=black, double, double distance between line centers=8pt] (b4) -- (c4);
                    
                        \node (iv) at (4.5, -1) {(iv)};
                    
                        \node[shape=circle,draw=black] (a5) at (6, 5) { };
                        \node at (a5) {$v_{i, j+1}$};
                        \node[shape=circle,draw=black] (d5) at (6, 3.5) { };
                        \node[shape=circle,draw=black] (c5) at (6, 1.5) { };
                        \node[shape=circle,draw=black] (b5) at (6, 0) {$v_{i, j}$}; 
                        \draw[-, thick, draw=black, double, double distance between line centers=8pt] (a5) -- (d5) -- (c5) -- (b5);
                    
                        \node (v) at (6, -1) {(v)};
                    
                        \node[shape=circle,draw=black] (a6) at (7.5, 5) { };
                        \node at (a6) {$v_{i, j+1}$};
                        \node[shape=circle,draw=black] (d6) at (7.5, 3.5) { };
                        \node[shape=circle,draw=black] (c6) at (7.5, 1.5) { };
                        \node[shape=circle,draw=black] (b6) at (7.5, 0) {$v_{i, j}$}; 
                    
                        \node (vi) at (7.5, -1) {(vi)};
                    
                        \node (left) at (-0.5, 0) { };
                        \node (right) at (8, 0) { };
                    
                    \end{tikzpicture}
            }
            \subcaption{Vertical configurations} 
            \label{fig:vertical_action}
        \end{minipage}
        \hfill
        \begin{minipage}{0.49\textwidth}
            \centering
            \resizebox{\textwidth}{!}
            {
                \begin{tikzpicture}

                    \tikzset{minimum size=34pt}
                    
                        \node[shape=circle,draw=black] (al1) at (0, 5) { };
                        \node at (al1) {\small $v_{i, j}$}; 
                        \node[shape=circle,draw=black] (ar1) at (2, 5) { };
                        \node at (ar1) {\small $v_{i+1, j}$}; 
                        \node (1) at (1, 4.0) {(i)};
                    
                        \node[shape=circle,draw=black] (al2) at (0, 2.5) { };
                        \node at (al2) {\small $v_{i, j}$}; 
                        \node[shape=circle,draw=black] (ar2) at (2, 2.5) { };
                        \node at (ar2) {\small $v_{i+1, j}$}; 
                        \draw[-, thick] (al2) -- (ar2);
                        \node (2) at (1, 1.5) {(ii)};
                    
                        \node[shape=circle,draw=black] (al3) at (0, 0) { };
                        \node at (al3) {\small $v_{i, j}$}; 
                        \node[shape=circle,draw=black] (ar3) at (2, 0) { };
                        \node at (ar3) {\small $v_{i+1, j}$}; 
                        \draw[-, thick, draw=black, double, double distance between line centers=8pt] (al3) -- (ar3);
                        \node (3) at (1, -1) {(iii)};
                    
                        \node (4) at (-3, 0) { };
                        \node (5) at (5.5, 0) { };
                    
                    \end{tikzpicture}
            }
            \subcaption{Horizontal configurations}
            \label{fig:horizontal_action}
        \end{minipage}
    \caption{The six possible vertical and three possible horizontal edge configurations that can appear in a minimal tour subgraph. (a)~Vertical edge configurations within each subaisle and (b)~horizontal edge configurations between adjacent aisles, with $(i)$--$(iii)$ denoting no edge, one edge, and two parallel edges, respectively}
    \label{fig:configurations}
    \end{figure}

The different sizes of rectangular warehouses are defined by the
number of cross-aisles:
\emph{single-block} ($n=2$),
\emph{two-block} ($n=3$),
and \emph{multi-block} ($n > 3$).
We index the aisles by $i=0,\ldots,m-1$ from left to right
and the horizontal cross-aisles by $j=0,1,\dots,n-1$ from bottom to top.
Accordingly, each aisle $i$ contains vertices $v_{i,0},v_{i,1},\dots,v_{i,n-1}$
on its $n$ horizontal cross-aisles, and $v_{i,0}$ and $v_{i,n-1}$ denote the boundary cross-aisles.
For any two vertices $v_1, v_2 \in V$,
let $m(v_1, v_2)$ denote the \emph{edge multiplicity},
i.e., the number of parallel edges directly between $v_1$ and $v_2$.

To state the proof of Proposition~\ref{prop:determinacy} precisely,
we formalize a local measure of how many horizontal edges
touch each aisle vertex.
Horizontal edges run only between adjacent aisles, so at $v_{i,j}$
the only possible horizontal connections are to $v_{i-1,j}$ and $v_{i+1,j}$.
Define the \emph{horizontal degree}
\begin{equation}
\label{eq:dH}
    d_H(v_{i,j}) := m(v_{i-1,j}, v_{i,j}) + m(v_{i,j}, v_{i+1,j}),
\end{equation}
with the convention that a term is $0$
when the corresponding neighbor
does not exist (i.e., for $i=0$ or $i=m-1$).

Following \citet{dunn2025double},
we call a double vertical edge \emph{connecting}
when both endpoints have $d_H>0$
(equivalently, when each endpoint is incident to at least one horizontal edge).
In particular, for rectangular warehouses,
\citet{dunn2025double} show that connecting
double vertical traversals are not
required to maintain connectivity of a minimal tour subgraph.

Finally, we spell out the six vertical configurations $(i)$--$(vi)$
in Figure~\ref{fig:vertical_action}, in order:
\begingroup
\renewcommand{\labelenumi}{(\roman{enumi})}
\begin{enumerate}
    \item \emph{1pass (single traversal).}
    One pass along the subaisle between vertices $v_{i,j}$ and $v_{i,j+1}$,
    visiting every pick in the segment.
    The tour enters at one end of the subaisle and exits at the other.
    \item \emph{Top.}
    Enter and exit through the upper cross-aisle $v_{i,j+1}$:
    collect picks reachable from that side and return to the same cross-aisle,
    leaving the lower portion of the subaisle untraveled.
    \item \emph{Bottom.}
    The mirror of $(ii)$ at the lower cross-aisle $v_{i,j}$:
    enter and exit there, leaving the upper portion untraveled.
    \item \emph{Gap.}
    The tour enters and exits at the lower cross-aisle $v_{i,j}$ and likewise
    enters and exits at the upper cross-aisle $v_{i,j+1}$,
    collecting every pick
    in the subaisle; the two excursions are chosen so as to maximize the vertical
    distance left untraveled between the outermost visited picks (a middle gap).
    \item \emph{2pass (double traversal).}
    The full subaisle is traversed twice (e.g.\ once in each direction along the segment).
    \item \emph{None (void).}
    No vertical travel in the subaisle: the tour does not enter it.
    This applies only when there are no pick locations to visit in that subaisle.
\end{enumerate}
\endgroup

The three horizontal configurations in Figure~\ref{fig:horizontal_action}
describe the multiplicity of horizontal edges between two adjacent aisles at a
fixed cross-aisle: $(i)$~no edge, $(ii)$~a single edge, and $(iii)$~two parallel
edges.
At an interior vertex $v_{i,j}$, each incident aisle-to-aisle link contributes
$0$, $1$, or $2$ to $d_H(v_{i,j})$ according to the configuration on that link,
so $d_H$ is the sum of two such terms; at a boundary aisle, only one link
appears in~\eqref{eq:dH}.
In particular $d_H(v_{i,j})$ can be zero, odd, or even, which is the distinction
used in the proof below.

A \emph{merged subaisle segment}
is a combined segment between $v_{i,j}$ and $v_{i,k}$
(with $k > j + 1$)
such that every intermediate vertex $v_{i,\ell}$
with $j < \ell < k$
satisfies $d_H(v_{i,\ell})=0$
(equivalently, has no incident horizontal edges).
In other words, the segment is entered and exited only at its endpoints.

\begin{lem}[Subaisle Merging]
\label{lem:merging}
The admissible vertical configurations for a single subaisle
apply to a merged subaisle segment.
\end{lem}

\begin{figure}[ht]
\centering
\begin{tikzpicture}[scale=1]

\def\vSpacing{2.5}
\def\hOffset{7.5}
\def\nodeSize{10mm}
\def\boxWidth{1.5}

\tikzset{
  hEdge/.style={-, draw=black, line width=1.2pt, dashed}
}

\pgfmathsetmacro{\yZero}{-3*2.5}
\pgfmathsetmacro{\yOne}{-2*2.5}
\pgfmathsetmacro{\yTwo}{-1*2.5}
\pgfmathsetmacro{\yThree}{0}

\pgfmathsetmacro{\nodeRadius}{0.5}

\pgfmathsetmacro{\boxBottomZero}{\yZero + \nodeRadius}
\pgfmathsetmacro{\boxTopZero}{\yOne - \nodeRadius}

\pgfmathsetmacro{\boxBottomOne}{\yOne + \nodeRadius}
\pgfmathsetmacro{\boxTopOne}{\yTwo - \nodeRadius}

\pgfmathsetmacro{\boxBottomTwo}{\yTwo + \nodeRadius}
\pgfmathsetmacro{\boxTopTwo}{\yThree - \nodeRadius}

\draw[hEdge] (-2.5, \yThree) -- (2.5, \yThree);
\draw[hEdge] (-2.5, \yZero) -- (2.5, \yZero);

\fill[gray!20] (-0.5*\boxWidth, \boxBottomZero) rectangle (0.5*\boxWidth, \boxTopZero);
\fill[gray!20] (-0.5*\boxWidth, \boxBottomOne) rectangle (0.5*\boxWidth, \boxTopOne);
\fill[gray!20] (-0.5*\boxWidth, \boxBottomTwo) rectangle (0.5*\boxWidth, \boxTopTwo);

\node[circle, draw=black, fill=white, minimum size=10mm] (Lv0) at (0, \yZero) {$v_{i,0}$};
\node[circle, draw=black, fill=white, minimum size=10mm] (Lv1) at (0, \yOne) {$v_{i,1}$};
\node[circle, draw=black, fill=white, minimum size=10mm] (Lv2) at (0, \yTwo) {$v_{i,2}$};
\node[circle, draw=black, fill=white, minimum size=10mm] (Lv3) at (0, \yThree) {$v_{i,3}$};

\node[draw=none, align=center] at (0, 1.5) 
    {(a) Individual subaisles};

\draw[hEdge] (\hOffset-2.5, \yThree) -- (\hOffset+2.5, \yThree);
\draw[hEdge] (\hOffset-2.5, \yZero) -- (\hOffset+2.5, \yZero);

\pgfmathsetmacro{\mergedBoxBottom}{\yZero + \nodeRadius}
\pgfmathsetmacro{\mergedBoxTop}{\yThree - \nodeRadius}
\fill[gray!20] (\hOffset-0.5*\boxWidth, \mergedBoxBottom) rectangle (\hOffset+0.5*\boxWidth, \mergedBoxTop);

\node[circle, draw=black, fill=white, minimum size=10mm] (Rv0) at (\hOffset, \yZero) {$v_{i,0}$};
\node[circle, draw=black, fill=white, minimum size=10mm] (Rv1) at (\hOffset, \yOne) {$v_{i,1}$};
\node[circle, draw=black, fill=white, minimum size=10mm] (Rv2) at (\hOffset, \yTwo) {$v_{i,2}$};
\node[circle, draw=black, fill=white, minimum size=10mm] (Rv3) at (\hOffset, \yThree) {$v_{i,3}$};

\node[draw=none, align=center] at (\hOffset, 1.5) 
    {(b) Merged segment};

\end{tikzpicture}
\vspace{0.5cm}
\caption{Illustration of subaisle merging (Lemma~\ref{lem:merging}).
Dashed lines show horizontal edges only at
the top and bottom vertices; gray shaded
areas illustrate the segments considered before and after merging.
(a) Individual subaisles between $v_{i,0}$ and $v_{i,3}$.
(b) Merged segment, where all subaisles between
$v_{i,0}$ and $v_{i,3}$ are treated as a single combined segment}
\label{fig:structural_merging}
\end{figure}

\begin{proof}[Proof of Lemma~\ref{lem:merging}]
By the definition of a merged subaisle segment,
every intermediate vertex between $v_{i,j}$ and $v_{i,k}$
has no incident horizontal edges.
Hence the segment can only be entered and exited via
$v_{i,j}$ and $v_{i,k}$, and there are no horizontal
connections that would require the tour to enter or exit
at intermediate points.
Given the admissible vertical configurations for a subaisle
(Figure~\ref{fig:vertical_action}),
the segment therefore behaves as a single unit,
and those same configurations apply to the merged segment as a whole.
Figure~\ref{fig:structural_merging} illustrates this concept,
showing how subaisles between $v_{i,0}$ and $v_{i,3}$ can be merged
when intermediate vertices $v_{i,1}$ and $v_{i,2}$ have no
incident horizontal edges.
\end{proof}

% ~~~~~~~~~~~~~~~~~~~~~~~~~~~~~~~~~~~~~~~~~~~~~~~~~~~~~~~~~~~~~~~~~

\section{Main Result}
\label{sec:main_result}

We now restrict our attention to
vertical configurations over merged subaisle
segments as permitted by Lemma~\ref{lem:merging}.
Within such merged segments,
\citet{dunn2025double} show that for rectangular warehouses,
connecting double edges (vertical configuration $(v)$)
are not required
to maintain connectivity of a minimal tour subgraph.
For the remainder of this section, we therefore
focus on minimal tour subgraphs in which configuration $(v)$ does not occur.
This restriction applies only to rectangular warehouses as
in non-rectangular geometries,
configuration $(v)$ may be needed
to keep the tour subgraph connected even when the degree-parity and
item-visiting constraints alone would allow an alternative pattern.

\begin{prop}[Deterministic Structure of Vertical Configurations]
\label{prop:determinacy}
Let $T$ be a minimal tour subgraph of a rectangular warehouse.
If all horizontal edges incident to the vertices
of some aisle are known, then the vertical edge configurations
of $T$ within that aisle are uniquely determined.
\end{prop}

\begin{figure}[ht]
\centering
\begin{tikzpicture}[scale=0.8, transform shape]

\def\nodeSize{10mm}
\def\nodeRadius{0.5}
\def\vSpacing{3.75}
\def\hSpacing{1.25}
\def\edgeLength{1.5}

\pgfmathsetmacro{\yZero}{-3*\vSpacing}
\pgfmathsetmacro{\yOne}{-2*\vSpacing}
\pgfmathsetmacro{\yTwo}{-1*\vSpacing}
\pgfmathsetmacro{\yThree}{0}

\pgfmathsetmacro{\grayYOne}{(\yZero + \yOne) / 2}
\pgfmathsetmacro{\grayYTwo}{(\yOne + \yTwo) / 2}
\pgfmathsetmacro{\grayYThree}{(\yTwo + \yThree) / 2}

\pgfmathsetmacro{\caseOneX}{0}
\pgfmathsetmacro{\caseTwoX}{4}
\pgfmathsetmacro{\caseThreeX}{8}
\pgfmathsetmacro{\caseFourX}{12}

\tikzset{
  sng/.style={-, line width=1pt}
}
\tikzset{
  dbl/.style={-, line width=1pt, double, double distance=4pt}
}

% Case 1
\node[circle, draw=black, fill=white, minimum size=\nodeSize] (C1v0) at (\caseOneX, \yZero) {$v_{i,0}$};
\node[circle, draw=black, fill=white, minimum size=\nodeSize] (C1v1) at (\caseOneX, \yOne) {$v_{i,1}$};
\node[circle, draw=black, fill=white, minimum size=\nodeSize] (C1v2) at (\caseOneX, \yTwo) {$v_{i,2}$};
\node[circle, draw=black, fill=white, minimum size=\nodeSize] (C1v3) at (\caseOneX, \yThree) {$v_{i,3}$};
\node[circle, draw=black, fill=gray!50, minimum size=\nodeSize] (C1g1) at (\caseOneX, \grayYOne) {};
\node[circle, draw=black, fill=gray!50, minimum size=\nodeSize] (C1g2) at (\caseOneX, \grayYTwo) {};
\node[circle, draw=black, fill=gray!50, minimum size=\nodeSize] (C1g3) at (\caseOneX, \grayYThree) {};
\draw[sng] (\caseOneX-\edgeLength, \yZero) -- (\caseOneX-\nodeRadius, \yZero);
\draw[sng] (\caseOneX+\nodeRadius, \yOne) -- (\caseOneX+\edgeLength, \yOne);
\draw[sng] (\caseOneX+\nodeRadius, \yTwo) -- (\caseOneX+\edgeLength, \yTwo);
\draw[sng] (\caseOneX-\edgeLength, \yThree) -- (\caseOneX-\nodeRadius, \yThree);
\draw[sng] (\caseOneX, \yZero+\nodeRadius) -- (\caseOneX, \grayYOne-\nodeRadius);
\draw[sng] (\caseOneX, \grayYOne+\nodeRadius) -- (\caseOneX, \yOne-\nodeRadius);
\draw[sng] (\caseOneX, \yTwo+\nodeRadius) -- (\caseOneX, \grayYThree-\nodeRadius);
\draw[sng] (\caseOneX, \grayYThree+\nodeRadius) -- (\caseOneX, \yThree-\nodeRadius);
\node[draw=none, align=center] at (\caseOneX, 1.2) {Case 1};

% Case 2
\node[circle, draw=black, fill=gray!50, minimum size=\nodeSize] (C2v0) at (\caseTwoX, \yZero) {$v_{i,0}$};
\node[circle, draw=black, fill=white, minimum size=\nodeSize] (C2v1) at (\caseTwoX, \yOne) {$v_{i,1}$};
\node[circle, draw=black, fill=white, minimum size=\nodeSize] (C2v2) at (\caseTwoX, \yTwo) {$v_{i,2}$};
\node[circle, draw=black, fill=white, minimum size=\nodeSize] (C2v3) at (\caseTwoX, \yThree) {$v_{i,3}$};
\node[circle, draw=black, fill=gray!50, minimum size=\nodeSize] (C2g1) at (\caseTwoX, \grayYOne) {};
\node[circle, draw=black, fill=gray!50, minimum size=\nodeSize] (C2g2) at (\caseTwoX, \grayYTwo) {};
\node[circle, draw=black, fill=gray!50, minimum size=\nodeSize] (C2g3) at (\caseTwoX, \grayYThree) {};
\draw[dbl] (\caseTwoX, \yZero+\nodeRadius) -- (\caseTwoX, \grayYOne-\nodeRadius);
\draw[dbl] (\caseTwoX, \grayYOne+\nodeRadius) -- (\caseTwoX, \yOne-\nodeRadius);
\draw[dbl] (\caseTwoX, \yOne+\nodeRadius) -- (\caseTwoX, \grayYTwo-\nodeRadius);
\draw[dbl] (\caseTwoX, \grayYTwo+\nodeRadius) -- (\caseTwoX, \yTwo-\nodeRadius);
\draw[dbl] (\caseTwoX, \yTwo+\nodeRadius) -- (\caseTwoX, \grayYThree-\nodeRadius);
\node[draw=none, align=center] at (\caseTwoX, 1.2) {Case 2};

% Case 3
\node[circle, draw=black, fill=white, minimum size=\nodeSize] (C3v0) at (\caseThreeX, \yZero) {$v_{i,0}$};
\node[circle, draw=black, fill=white, minimum size=\nodeSize] (C3v1) at (\caseThreeX, \yOne) {$v_{i,1}$};
\node[circle, draw=black, fill=white, minimum size=\nodeSize] (C3v2) at (\caseThreeX, \yTwo) {$v_{i,2}$};
\node[circle, draw=black, fill=white, minimum size=\nodeSize] (C3v3) at (\caseThreeX, \yThree) {$v_{i,3}$};
\node[circle, draw=black, fill=gray!50, minimum size=\nodeSize] (C3g1) at (\caseThreeX, \grayYOne) {};
\node[circle, draw=black, fill=gray!50, minimum size=\nodeSize] (C3g2) at (\caseThreeX, \grayYTwo) {};
\node[circle, draw=black, fill=gray!50, minimum size=\nodeSize] (C3g3) at (\caseThreeX, \grayYThree) {};
\draw[dbl] (\caseThreeX-\edgeLength, \yZero) -- (\caseThreeX-\nodeRadius, \yZero);
\draw[dbl] (\caseThreeX+\nodeRadius, \yZero) -- (\caseThreeX+\edgeLength, \yZero);
\draw[dbl] (\caseThreeX-\edgeLength, \yThree) -- (\caseThreeX-\nodeRadius, \yThree);
\draw[dbl] (\caseThreeX+\nodeRadius, \yThree) -- (\caseThreeX+\edgeLength, \yThree);
\draw[dbl] (\caseThreeX, \yThree-\nodeRadius) -- (\caseThreeX, \grayYThree+\nodeRadius);
\draw[dbl] (\caseThreeX, \yZero+\nodeRadius) -- (\caseThreeX, \grayYOne-\nodeRadius);
\draw[dbl] (\caseThreeX, \grayYOne+\nodeRadius) -- (\caseThreeX, \yOne-\nodeRadius);
\draw[dbl] (\caseThreeX, \yOne+\nodeRadius) -- (\caseThreeX, \grayYTwo-\nodeRadius);
\node[draw=none, align=center] at (\caseThreeX, 1.2) {Case 3};

% Case 4
\node[circle, draw=black, fill=white, minimum size=\nodeSize] (C4v0) at (\caseFourX, \yZero) {$v_{i,0}$};
\node[circle, draw=black, fill=white, minimum size=\nodeSize] (C4v1) at (\caseFourX, \yOne) {$v_{i,1}$};
\node[circle, draw=black, fill=white, minimum size=\nodeSize] (C4v2) at (\caseFourX, \yTwo) {$v_{i,2}$};
\node[circle, draw=black, fill=white, minimum size=\nodeSize] (C4v3) at (\caseFourX, \yThree) {$v_{i,3}$};
\node[circle, draw=black, fill=gray!50, minimum size=\nodeSize] (C4g1) at (\caseFourX, \grayYOne) {};
\node[circle, draw=black, fill=gray!50, minimum size=\nodeSize] (C4g2) at (\caseFourX, \grayYTwo) {};
\node[circle, draw=black, fill=gray!50, minimum size=\nodeSize] (C4g3) at (\caseFourX, \grayYThree) {};
\draw[sng] (\caseFourX-\edgeLength, \yThree) -- (\caseFourX-\nodeRadius, \yThree);
\draw[sng] (\caseFourX+\nodeRadius, \yThree) -- (\caseFourX+\edgeLength, \yThree);
\draw[dbl] (\caseFourX, \yThree-\nodeRadius) -- (\caseFourX, \grayYThree+\nodeRadius);
\draw[dbl] (\caseFourX, \grayYThree-\nodeRadius) -- (\caseFourX, \yTwo+\nodeRadius);
\draw[dbl] (\caseFourX, \yTwo-\nodeRadius) -- (\caseFourX, \grayYTwo+\nodeRadius);
\draw[dbl] (\caseFourX, \grayYTwo-\nodeRadius) -- (\caseFourX, \yOne+\nodeRadius);
\draw[dbl] (\caseFourX, \yOne-\nodeRadius) -- (\caseFourX, \grayYOne+\nodeRadius);
\node[draw=none, align=center] at (\caseFourX, 1.2) {Case 4};

\end{tikzpicture}
\vspace{0.5cm}
\caption{Schematic illustration of the four cases in Proposition~\ref{prop:determinacy}
(vertex labels and depot placement are representative; Case~2 places the depot at an
end cross-aisle of the segment, as in the proof, with $(ii)$ or $(iii)$ depending on
whether that end is the top or bottom).
Case 1: Odd $d_H$ at vertices paired along the aisle, with single vertical edges
between consecutive pairs.
Case 2: $d_H\equiv 0$ on the aisle, with double edges from the depot end toward the picks.
Case 3: Even $d_H$ at both endpoints of the segment, with double edges from $v_{i,3}$
to the uppermost pick location and from $v_{i,0}$ to the second uppermost pick location.
Case 4: $d_H=0$ at one endpoint and positive even $d_H$ at the other, with a double edge from $v_{i,3}$
to the lowermost pick location.
Gray nodes represent pick locations (including the depot where shown)}
\label{fig:structural_determinacy}
\end{figure}

\noindent
Figure~\ref{fig:structural_determinacy} schematically
illustrates Cases~1--4 developed in the proof below;
readers may consult it alongside each paragraph.

\begin{proof}[Proof of Proposition~\ref{prop:determinacy}]

Given the set of horizontal edges
incident to a particular aisle in
a minimal tour subgraph, Theorem~\ref{thm:subtour} requires
that vertical configurations ensure all items $P$ are visited,
$T$ is connected, and every vertex in $T$ has even degree.

The key insight is that once horizontal edges are fixed,
the degree parity at each vertex is determined,
and the Eulerian condition (even degree at all vertices)
uniquely constrains the vertical edge configurations.
Merging subaisles according to Lemma~\ref{lem:merging},
we first resolve vertices with odd horizontal degree along the entire
aisle (Case~1).
The remaining cases partition the parity patterns at the two endpoints
of each merged segment (Cases~2--4).

\paragraph{Case 1: Vertices with odd horizontal degree.}
Between any two aisles, horizontal edges appear in even multiplicity
\citep{ratliff1983order}, so along a fixed aisle the vertices with odd~$d_H$
are even in number (possibly none).
For Eulerian parity, each odd-$d_H$ vertex needs an odd count of incident
vertical edges and each even-$d_H$ vertex an even count.
In Figure~\ref{fig:vertical_action}, only configuration~$(i)$ contributes a
single vertical edge at a subaisle endpoint; double-edge vertical patterns add
even multiplicity there and cannot offset odd~$d_H$.

Pair odd-$d_H$ vertices consecutively from bottom to top and connect each pair
by one~$(i)$.
This is an aisle-wide construction: the path may cross several merged segments
and intermediate vertices with $d_H=0$ or even~$d_H$, each of which picks up
two vertical incidences along the walk and remains even overall.
A nested pairing (two odds joined with another odd between them on the aisle)
forces extra vertical travel that visits no new picks and, under uniform
cross-aisle spacing, strictly more length than consecutive pairing.
Hence a minimal tour uses~$(i)$ exactly between successive odd-$d_H$ vertices,
and Theorem~\ref{thm:subtour} holds.

\medskip

\noindent
Cases~2--4 complete the vertical pattern: Case~2 is the degenerate situation
with $d_H\equiv 0$ at every aisle vertex of $T$; Cases~3 and~4 classify each
remaining merged segment by the pair of $d_H$ values at its endpoints.

\paragraph{Case 2: \texorpdfstring{$d_H\equiv 0$ on the aisle}{Zero horizontal degree on the aisle}.}
This occurs only if the depot and all items are in the same aisle,
meaning no horizontal travel is needed.
We may assume that $T$ meets only this aisle, since no pick lies elsewhere
and horizontal edges are absent.
The minimal configuration must connect the depot
to all pick locations using vertical edges only.
Then $d_H(v_{i,j})=0$ for every aisle vertex $v_{i,j}$ of $T$, and
Eulerian parity forces $T$ to use a single vertical cycle
within this aisle.
Because horizontal movement is absent, minimality implies we
must reach the farthest pick location and return, so the minimal
configuration adds vertical edges connecting the depot to the
farthest item.
If the depot is in the bottom cross-aisle, this corresponds
to configuration~$(iii)$, and
if it is in the top cross-aisle ($j=n-1$),
to configuration~$(ii)$.

\paragraph{Case 3: \texorpdfstring{Even $d_H$ at both endpoints}{Even horizontal degree at both endpoints}.}
Let $v_{i,j}$ and $v_{i,k}$ be the endpoints of a merged segment.
When $d_H(v_{i,j})$ and $d_H(v_{i,k})$ are both even,
the horizontal contribution to total degree is already even at both endpoints.
No configuration $(i)$ is admissible here:
it would add exactly one vertical edge at each endpoint of the segment,
making the total degree at each endpoint odd because even $d_H$ plus one
vertical edge is odd.
The vertical configuration must therefore use double edges
to visit all pick locations while preserving even total degree.
If we enforce only the degree-parity and coverage constraints,
two parity-feasible patterns may arise, i.e.\
\emph{gap} $(iv)$ and \emph{double traversal} $(v)$.
The gap pattern $(iv)$ covers all required picks while maximizing the
untraveled middle section, hence minimizing vertical travel among
parity-feasible options.
The connecting double traversal $(v)$ traverses that middle section as
well, so it is strictly longer than $(iv)$ without
visiting additional required picks.
In rectangular warehouses, \citet{dunn2025double} show that such
connecting double edges are not required to maintain connectivity of
a minimal tour subgraph.
Therefore, within the rectangular setting considered in this paper,
configuration $(iv)$ is the unique minimal feasible choice.
In non-rectangular settings, the connectivity argument can break down,
and both configurations $(iv)$ and $(v)$ would need to be considered.

\paragraph{Case 4: \texorpdfstring{$d_H=0$ at one endpoint and even $d_H>0$ at the other}
{One endpoint zero horizontal degree, the other positive even}.}
When one endpoint of the merged segment has $d_H=0$ and the other has
even $d_H>0$,
the $d_H=0$ endpoint cannot contribute to
the connectivity of the tour via cross-aisles.
The segment must therefore be accessed entirely from the
endpoint with $d_H>0$.
Since that endpoint already has even $d_H$,
a double edge is required to visit all pick locations
and return without changing its total-degree parity.
The minimal such configuration extends a double edge from
the endpoint with $d_H>0$ to the farthest pick location
and returns, corresponding to configuration $(ii)$
if that endpoint is at the top of the segment,
or $(iii)$ if it is at the bottom.
Any shorter double edge would fail to reach all pick locations and
any longer one would traverse beyond the farthest pick location
unnecessarily.
Thus the configuration is uniquely determined.

\medskip

\noindent
These four cases exhaust the relevant $d_H$ patterns:
Case~1 pairs every vertex with odd $d_H$ along the aisle;
Case~2 is the all-zero pattern $d_H\equiv 0$ on the aisle;
Cases~3 and~4 cover merged segments whose endpoints have two even values of
$d_H$ (Case~3), or one value $d_H=0$ and one positive even $d_H$ (Case~4, in
either order along the segment).
In each case, the vertical edge configuration is uniquely determined
by the requirement to minimize travel distance while satisfying
the Eulerian conditions and ensuring all pick locations are visited.
Therefore, once horizontal edges are fixed, vertical configurations
are deterministic.

\end{proof}

Figure~\ref{fig:structural_determinacy} previews the geometry of each case.
Table~\ref{tab:vertical_configuration_summary} lists the vertical configuration
forced by each horizontal-degree pattern in the proof of
Proposition~\ref{prop:determinacy}.
Once that pattern is fixed at segment endpoints, Eulerian parity and minimality
leave exactly one of $(i)$--$(iv)$, as in the table.
\begin{table}[ht]
\centering
\small
\setlength{\tabcolsep}{4pt}
\begin{tabular}{@{}l p{0.44\linewidth} l@{}}
\hline
& \multicolumn{1}{l}{Horizontal-degree pattern (merged segment or aisle)} & Vertical config. \\
\hline
Case 1 & Odd $d_H$ paired along aisle (see proof) & $(i)$ \\
Case 2 & $d_H\equiv 0$ on aisle (single-aisle) & $(ii)$ or $(iii)$ (depot) \\
Case 3 & Even $d_H$ at both endpoints & $(iv)$ \\
Case 4 & One endpoint $d_H{=}0$, other even $d_H>0$ & $(ii)$ or $(iii)$ ($d_H{>}0$) \\
\hline
\end{tabular}
\caption{Unique vertical configurations implied by horizontal degree $d_H$ at merged-segment endpoints (Cases~2--4) and along the aisle (Case~1)}
\label{tab:vertical_configuration_summary}
\end{table}

\vspace{0.5cm}

% ~~~~~~~~~~~~~~~~~~~~~~~~~~~~~~~~~~~~~~~~~~~~~~~~~~~~~~~~~~~~~~~~~

\section{Conclusion}
\label{sec:conclusion}

We have demonstrated that in rectangular warehouses,
the vertical edge configurations
in a minimal tour subgraph are uniquely determined by
the horizontal edge structure.
This deterministic relationship provides a deeper
understanding of optimal picker route structure
and eliminates the need to explore
vertical and horizontal configurations jointly.
Crucially, this characterization relies on our rectangular-warehouse
connectivity result \citep{dunn2025double}, which permits excluding
the double traversal vertical pattern $(v)$ when
determining the minimal vertical edges from the horizontal structure.

As existing dynamic programming algorithms
for the picker routing problem
alternate between horizontal and vertical stages
\citep{ratliff1983order, roodbergen2001routing, pansart2018exact},
this result implies that the vertical stages
can be replaced by a deterministic inference step,
potentially reducing the number of stages required.
More broadly, methods exist that formulate routing with edge
configurations as decision variables in rectangular warehouses
\citep{goeke2021modeling}; our result implies that such optimization models
can instead treat horizontal edges as the sole decision variables,
with vertical edges inferred directly,
reducing the combinatorial complexity of the problem
and providing a clearer framework for developing
exact methods for warehouse layouts of any size.

% ~~~~~~~~~~~~~~~~~~~~~~~~~~~~~~~~~~~~~~~~~~~~~~~~~~~~~~~~~~~~~~~~~

\clearpage

\section*{Statements and Declarations}

\paragraph{Funding}
George Dunn was supported by an Australian Government
Research Training Program (RTP) Scholarship. No other
funding was received for this work.

\paragraph{Competing interests}
The authors have no competing interests to declare
that are relevant to the content of this article.

\paragraph{Data availability}
This work is purely theoretical and does not use
empirical datasets. No data were generated or analysed.

\paragraph{Author contributions}
George Dunn led the conceptualization, proofs,
and drafting of the manuscript.
The co-authors provided supervision, guidance on the research direction and
presentation, and critical review of the results.

\bibliography{bibliography}

\end{document}